\documentclass[12pt]{amsart}
\usepackage{amsmath}
\usepackage{mathrsfs}
\usepackage{psfrag}
\usepackage{graphicx,amssymb,booktabs, verbatim, mathtools}
\usepackage[noadjust]{cite}
\usepackage{xspace}

\usepackage[pdfusetitle, pdfpagelabels, plainpages=false,
  bookmarks, bookmarksnumbered
]{hyperref} 
\usepackage{bookmark}
\usepackage{subfigure}
\newif\iffloattoend

\iffloattoend
\usepackage[figuresonly, nomarkers, nolists, heads]{endfloat}

\fi

\newcommand{\RR}{\mathbb R}
\newcommand{\ZZ}{\mathbb Z}
\newcommand{\CC}{\mathbb C}
\newcommand{\FF}{\mathbb F}
\newcommand{\PP}{\mathbb{P}}

\newcommand{\GL}{\mathrm{GL}}
\newcommand{\PGL}{\mathrm{PGL}}

\newcommand{\sets}[1]{[\![#1]\!]}

\newcommand{\cT}{\mathcal{T}}

\newcommand{\sT}{\mathscr{T}}

\newcommand{\bino}[2]{\binom{#1}{#2}}

\swapnumbers

\theoremstyle{plain}
\newtheorem{thm}[subsection]{Theorem}
\newtheorem{lem}[subsection]{Lemma}

\newtheorem{prop}[subsection]{Proposition}

\newtheorem{cor}[subsection]{Corollary}
\theoremstyle{definition}
\newtheorem{defn}[subsection]{Definition}
\newtheorem{remark}[subsection]{Remark}

\newtheorem{ex}[subsection]{Example}

\DeclareRobustCommand{\stirling}{\genfrac\{\}{0pt}{}}

\DeclareMathOperator{\des}{\mathrm{des}}

\catcode`\@=11
\def\@secnumfont{\bfseries}
\catcode`\@12

\begin{document}

\title[Tiered trees, weights, and $q$-Eulerian numbers]{Tiered trees, weights, and $q$-Eulerian numbers}

\author{William Dugan}
\address{Department of Mathematics and Statistics\\University of
Massachusetts\\Amherst, MA 01003-9305}
\email{wdugan@umass.edu}

\author{Sam Glennon}
\address{Department of Mathematics and Statistics\\University of
Massachusetts\\Amherst, MA 01003-9305}
\email{glennon@math.umass.edu}

\author{Paul E. Gunnells}
\address{Department of Mathematics and Statistics\\University of
Massachusetts\\Amherst, MA 01003-9305}
\email{gunnells@math.umass.edu}

\author{Einar Steingr\'\i msson}
\address{Department of Computer and Information Sciences\\
University of Strathclyde\\Livingstone Tower\\
26 Richmond Street\\
Glasgow, G1 1XH, UK}
\email{einar@alum.mit.edu}

\date{8 Feb 2017}
\thanks{This work was partially supported by NSF grant DMS 1501832.
We thank the referees for helpful comments}



\begin{abstract}
\emph{Maxmin trees} are labeled trees with the property that each
vertex is either a local maximum or a local minimum.  Such trees were
originally introduced by Postnikov \cite{postnikov}, who gave a
formula to count them and different combinatorial interpretations for
their number.  In this paper we generalize this construction and
define \emph{tiered trees} by allowing more than two classes of
vertices.  Tiered trees arise naturally when counting the absolutely
indecomposable representations of certain quivers, and also when one
enumerates torus orbits on certain homogeneous varieties.  We define a
notion of weight for tiered trees and prove bijections between various
weight 0 tiered trees and other combinatorial objects; in particular
order $n$ weight 0 maxmin trees are naturally in bijection with
permutations on $n-1$ letters.  We conclude by using our weight
function to define a new $q$-analogue of the Eulerian numbers.
\end{abstract}

\maketitle

\section{Introduction}\label{s:intro}

\subsection{} Let $T$ be a tree with vertices labeled by the ordered
set $\{1,\dotsc ,n \}$.  We say $T$ is a \emph{maxmin} tree if for any
vertex $v$, the labels of its neighbors are either all less than or
all greater than that of $v$.  Such trees were introduced by Postnikov
\cite{postnikov}; he called them \emph{intransitive trees}, since they
satisfy the property that for any triple $1\leq i<j<k \leq n$, the
pairs $\{i,j \}$ and $\{j,k \}$ cannot both be edges of $T$.  These
trees first appeared in the study of hypergeometric systems attached
to root systems \cite{ggp}, and were later connected to a variety of
combinatorial objects:
\begin{itemize}
\item regions of the Linial hyperplane arrangement (the affine
arrangement in $\RR^{n}$ defined by the equations $x_{i}-x_{j} = 1$,
$1\leq i<j\leq n$);
\item local binary search trees (labeled plane binary trees with the
property that every left child has a smaller label than its parent and
every right child has a larger label than its parent);
\item and semiacyclic tournaments (directed graphs on the set
$\{1,\dotsc ,n \}$ such that in every directed cycle,
there are more edges $(i,j)$ with $i<j$ than with $i > j$).
\end{itemize}
For more details, we refer to \cite{postnikov, stanley-pnas, cfv,
forge, ps}.

\subsection{} In this paper, we consider a generalization of
Postnikov's trees called \emph{tiered trees}.  Instead of two classes
of vertices, maxima and minima, we allow more classes.  More
precisely, a tiered tree with $m\geq 2$ tiers is a tree on the vertex set
$V=\{1,\dotsc ,n \}$ with a map $t$ from $V$ to an ordered set with
$m$ elements.  We require that if $v$ is a vertex adjacent to $v'$
with $v>v'$, then $t (v)>t (v')$.  

Tiered trees naturally arise in two a priori unrelated geometric
counting problems \cite{supernova}:
\begin{itemize}
\item counting absolutely irreducible representations of the supernova
quivers (quivers arising in the study of the moduli
spaces of certain irregular meromorphic connections on trivial
bundles over $\PP^{1}$ \cite{boalch}), and
\item counting certain torus orbits on partial flag varieties of type
$A$ over finite fields, namely those orbits with trivial stabilizers.
\end{itemize}
These contexts also motivate defining a weight function on tiered
trees with values in $\ZZ_{\geq 0}$.  We define weights of
tiered trees and show that the weights are related to the Tutte
polynomials of certain graphs (Theorem \ref{thm:wtisactivity}).  We
show that weight $0$ trees of the two extreme types of tiering
functions---maxmin trees, which have two tiers, and fully tiered
trees, in which every vertex lies on a different tier---are related to
other combinatorial objects, namely permutations sorted by their
descents and the complete nonambiguous trees of
Aval--Boussicault--Bouvel--Silimbani \cite{nat}.  Finally we use the
weights of maxmin trees to define a weight for permutations, which
leads to a version of $q$-Eulerian numbers different from those
studied by Carlitz, Stanley, and Shareshian--Wachs.

\subsection{} Here is a guide to the sections of this paper.  In \S
\ref{s:tt} we define tiered trees and their weights.  We also discuss
the connection between the weight and geometry in the case of counting
torus orbits.  In \S  \ref{s:ec} we discuss the enumeration of tiered
trees.  The next two sections give combinatorial interpretations of
the two extreme cases of weight zero trees: \S \ref{s:en} treats
weight zero maxmin trees, and \S \ref{s:nat} treats weight zero fully
tiered trees.  Finally in \S \ref{s:wp} we use weighted maxmin trees
to define a notion of weight for permutations and our $q$-Eulerian
polynomials.  We conclude with a discussion of some open questions.

\section{Tiered trees and weights}\label{s:tt}

\subsection{} For any $n$ let $\sets{n}$ be the finite set $\{1,\dotsc
,n \}$.  Let $G$ be a graph with vertices $V$ labeled by $\sets{n}$
and let $m\geq 2$ be an integer.  We say $G$ is \emph{tiered with $m$
levels} if there is a surjective function $t\colon V \rightarrow
\sets{m}$ such that $t (v)\not = t (v')$ for any adjacent vertices
$v,v'$, and such that if $v,v'$ are adjacent and $v>v'$, then $t (v) >
t (v')$.  We call $t$ a \emph{tiering function} and say $G$ is
\emph{tiered by $t$}.  Given $i\in \sets{m}$, we say that the vertices
$v$ with $t (v) = i$ are at \emph{tier $i$}.
Any tiered graph determines a composition $p = (p_{1}, \dotsc ,p_{m})$
of $|V|$ into $m$ parts by putting $p_{k} = |t^{-1} (k)|$.  We call
$p$ the \emph{tier type}.

\begin{defn}\label{def:tt}
A \emph{tiered tree} $T$ with tiering function $t$ is a labeled tree
tiered by $t$.
\end{defn}

Thus a tiered tree with $m$ levels can be visualized as a tree with
its vertices placed at $m$ different heights.  Vertices at the same
height can never be adjacent, and if two vertices at different heights
are adjacent, we require the label of the higher vertex to be larger
than that of the lower vertex.  Note that this does not mean that the
maximal label must occupy the top tier, nor does it mean that the
minimal label must occupy the bottom tier.  Figure
\ref{fig:exampletree} shows a tree tiered with four levels and with
tier type $(6,4,2,3)$.

We say a tree is \emph{fully tiered} if $m=|V|$, and is a \emph{maxmin
tree} if $m=2$.  For maxmin trees, it is natural to call the vertices
in the top tier \emph{maxima} and those in the bottom tier
\emph{minima}.  Figure \ref{fig:4verts} shows all the maxmin trees on
$4$ vertices, grouped according to the number of maxima they have.
Similarly, Figure \ref{fig:ft3} shows the fully tiered trees on $3$
vertices. 

\begin{figure}[htb]
\psfrag{1}{$\scriptstyle 1$}
\psfrag{2}{$\scriptstyle 2$}
\psfrag{3}{$\scriptstyle 3$}
\psfrag{4}{$\scriptstyle 4$}
\psfrag{5}{$\scriptstyle 5$}
\psfrag{6}{$\scriptstyle 6$}
\psfrag{7}{$\scriptstyle 7$}
\psfrag{8}{$\scriptstyle 8$}
\psfrag{9}{$\scriptstyle 9$}
\psfrag{a}{$\scriptstyle A$}
\psfrag{b}{$\scriptstyle B$}
\psfrag{c}{$\scriptstyle C$}
\psfrag{d}{$\scriptstyle D$}
\psfrag{e}{$\scriptstyle E$}
\psfrag{f}{$\scriptstyle F$}
\begin{center}
\includegraphics[scale=0.35]{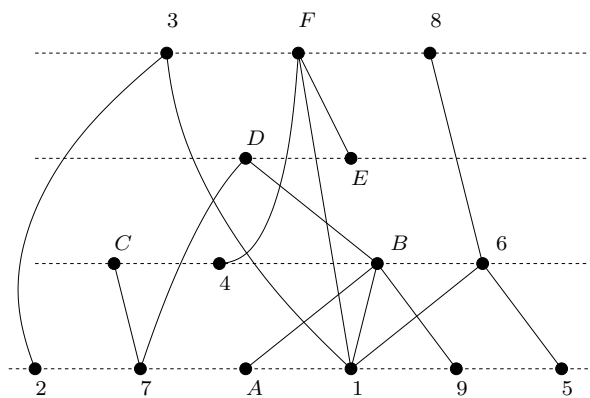}
\end{center}
\caption{A tiered tree with 4 levels and 15 vertices.  The tier type
is $(6,4,2,3)$.\label{fig:exampletree}}
\end{figure}

\begin{figure}[htb]
\psfrag{1}{$\scriptstyle 1$}
\psfrag{2}{$\scriptstyle  2$}
\psfrag{3}{$\scriptstyle  3$}
\psfrag{4}{$\scriptstyle 4$}
\begin{center}
\includegraphics[scale=0.3]{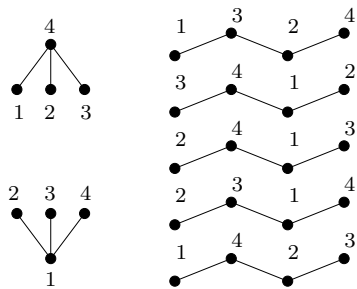}
\end{center}
\caption{All maxmin trees on $4$ vertices.\label{fig:4verts}}
\end{figure}

\begin{figure}[htb]
\psfrag{1}{$\scriptstyle 1$}
\psfrag{2}{$\scriptstyle  2$}
\psfrag{3}{$\scriptstyle  3$}
\psfrag{4}{$\scriptstyle 4$}
\begin{center}
\includegraphics[scale=0.3]{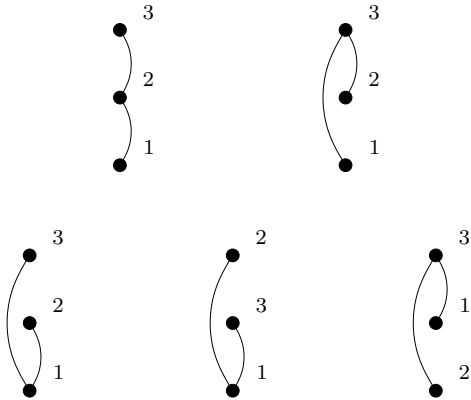}
\end{center}
\caption{All fully tiered trees on $3$ vertices.\label{fig:ft3}}
\end{figure}

For any tier type $p$ let $\cT_{p}$ be the corresponding set of tiered
trees, and let $\cT = \cup_{p} \cT_{p}$.  Our goal now is 
to define a weight function $w \colon \cT \rightarrow \ZZ_{\geq 0}$.

\begin{defn}\label{def:weightoftt}
Let $T\in \cT$ have tiering function $t$.  The \emph{weight} $w (T)$
of $T$ is defined as follows:
\begin{enumerate}
\item If $|T|=1$, i.e.~if $T$ consists of a single vertex, we put
$w (T) = 0$.
\item If $|T|>1$, let $v$ be the vertex with the smallest label, and
let $T_{1}$, \dots , $T_{l}$ be the connected components of the forest
obtained by deleting $v$ from $T$.  We put
\[
w (T) = \sum_{i=1}^{l} (w_{i} + w (T_{i})), 
\] 
where the integers $w_{i}$ are defined as follows.  For each $T_{i}$,
let $u_{i}\in T_{i}$ be the vertex that was connected to $v$.  We have
$t(u_i) > t(v)$ since $v$ had the minimal label.  We define $w_i$ to be the cardinality
of the set 
\begin{equation}\label{eq:Ri}
R_i = \{u_{j}\in T_i \mid \text{$t(u_{j}) > t(v)$ and $u_j < u_i$}\}.
\end{equation}
Thus $w_{i}$ records the (zero-indexed) position of the vertex $u_{i}$
in the ordered list of all those vertices in $T_{i}$ that \emph{could} have been
connected to $v$.
\end{enumerate}
\end{defn}

\begin{ex}\label{ex:exampletree}
We give a detailed example of computing the weight of a tiered tree by
computing $w (T)$ for the tree $T$ in Figure \ref{fig:exampletree}.
First we delete $1$ and obtain the forest $\{T_{i}\mid 1\leq i\leq 4
\}$ shown in Figure \ref{fig:exampleweight}.  We have $w (T) =
\sum_{i=1}^{4}(w_{i}+w (T_{i}))$, and we consider each component in turn:
\begin{enumerate}
\item We have $w (T_{1})=0$, since all tiered trees with only two
vertices have weight $0$.  Also $w_{1}=0$ since only vertex $3$ can be
joined to $1$.
\item We have $w (T_{2}) = 1$, since the minimal vertex $4$ is
connected to $F$ and could have been connected to $E$.  We have $w_{2}
= 2$, since $1$ could have been connected to any of $4<E<F$, and is
connected to $F$.
\item This component is the most complicated.  If we delete $7$, we
see that the tree $\{C \}$ has weight $0$ as does the tree $T' = \{9,A,B,D
\}$.  Furthermore $7$ is joined in $T'$ to $D$ instead of $B$.  Thus
altogether we have $w (T_{3}) = 1$.  Finally $1$ is joined to $B$, the
smallest possible vertex it could have been connected to, so $w_{3}=0$.
\item Finally $w (T_{4}) = 0$, and $1$ is joined to the smallest
possible vertex $6$, so $w_{4}=0$.
\end{enumerate}
Taking all these contributions together, we find 
\[
w (T) = 0+0+2+1+1+0+0 = 4.
\]
\end{ex}

\begin{ex}\label{ex:smalltrees}
One can check that all the trees in Figure
\ref{fig:4verts} have weight $0$ except for the tree on the bottom
right, which has weight $1$.  Indeed, after deleting $1$ from this
tree, we obtain a single tree $T_{1}$ with vertices $\{2,3,4 \}$.  We
have $w (T_{1}) = 0$ and $w_{1} = 1$, since $1$ was originally
connected to $4$ and not $3$.  It is easy to check that there is a
unique maxmin tree on $n$ vertices with $1$ maximum and a unique one
with $n-1$ maxima, and that the weights of both of these trees are
$0$.  This accounts for all the trees in Figure \ref{fig:4verts}.

Similarly, the weights of the trees in Figure \ref{fig:ft3} are
all $0$ except for the upper right tree, which has weight $1$.   In
this case, deleting $1$ yields the weight $0$ tree with vertices
$\{2,3 \}$, and again $w_{1}=1$ since $1$ could have been connected to
$2$ instead of $3$.
\end{ex}

\begin{ex}\label{ex:polys}
Table
\ref{tab:qpolys} gives a short table of the polynomials 
\[
P_{p} (q) = \sum_{T \in \cT_{p}}q^{w (T)}
\]
for various tier types $p$.  The apparent coincidences, such as
$P_{(1,1,1)} = P_{(2,2)}$ and $P_{(1,1,2)} = P_{(2,3)}$ can be
explained geometrically via the material in \S \ref{ss:geometric}.
\end{ex}

\begin{figure}[htb]
\psfrag{1}{$\scriptstyle 1$}
\psfrag{2}{$\scriptstyle 2$}
\psfrag{3}{$\scriptstyle 3$}
\psfrag{4}{$\scriptstyle 4$}
\psfrag{5}{$\scriptstyle 5$}
\psfrag{6}{$\scriptstyle 6$}
\psfrag{7}{$\scriptstyle 7$}
\psfrag{8}{$\scriptstyle 8$}
\psfrag{9}{$\scriptstyle 9$}
\psfrag{a}{$\scriptstyle A$}
\psfrag{b}{$\scriptstyle B$}
\psfrag{c}{$\scriptstyle C$}
\psfrag{d}{$\scriptstyle D$}
\psfrag{e}{$\scriptstyle E$}
\psfrag{f}{$\scriptstyle F$}
\psfrag{t1}{$\scriptstyle T_1$}
\psfrag{t2}{$\scriptstyle T_2$}
\psfrag{t3}{$\scriptstyle T_3$}
\psfrag{t4}{$\scriptstyle T_4$}
\begin{center}
\includegraphics[scale=0.35]{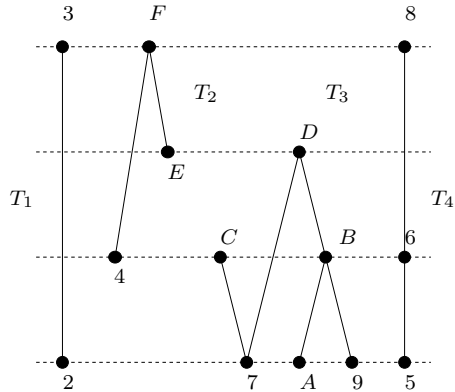}
\end{center}
\caption{The forest arising in computing the weight of the tree in
Figure \ref{fig:exampletree}.\label{fig:exampleweight}}
\end{figure}

\subsection{} Let $T$ be a tiered tree with respect to $t$.  The pair
$(T,t)$ determines a unique maximal tiered graph $K_{t}$ containing
$T$ as a spanning tree.  We simply add all edges to $T$ between
vertices $v,v'$ with $v>v'$ and $t (v)>t (v')$.  We call $K_{t}$ the
\emph{complete tiered graph $K_{t}$ of order $n$ with respect to $t$}.
By construction $K_{t}$ is connected.  It turns out that the weight of
$T$ can be interpreted via the Tutte polynomial of $K_{t}$.  We recall
the definition of this polynomial in the form we need
(cf.~\cite[Exercise 15.15]{godsil}).

Let $G$ be a graph with a total ordering of its edge set, let
$T\subset G$ be a spanning tree, and let $e$ be an edge of $G$.  We
say that $e $ is \emph{internally active} with respect to $T$ if it is
contained in $T$ and is the least element in the cut of $G$ determined
by $e$ and $T$.  We say that $e$ is \emph{externally active} with
respect to $T$ if it is not contained in $T$ and if it is the least
element in the unique cycle determined by $T\cup e$.  The number of
edges internally (respectively, externally) active with respect to $T$
is called the \emph{internal} (resp., \emph{external}) \emph{activity}
of $T$ in $G$.

\begin{defn}\label{def:tutte}
The \emph{Tutte polynomial} $T_{G} (x,y)$ is defined by 
\[
T_{G} (x,y) = \sum_{i,j} t_{ij}x^{i}y^{j},
\]
where $t_{ij}$ is the number of spanning trees of $G$ of internal
activity $i$ and external activity $j$.
\end{defn}

\begin{thm}\label{thm:wtisactivity}
Let $T$ be a tiered tree, and let $K_{t}$ be the complete tiered graph
containing it.  Then the weight $w (T)$ is equal to its external
activity in $K_{t}$, where the edges of $K_{t}$ are ordered
lexicographically by their endpoints.
\end{thm}
  
\begin{proof} We first prove that, in the algorithm described in
Definition \ref{def:weightoftt}, $w_i$ is equal to the number of
externally active edges in $K_t$ that connect $v$ to an element of
$T_i$.

To see this, let $u$ be a vertex satisfying 
\[
 u \in \{u' \in T_i \mid \text{$t(u') > t(v)$ and $ u' \neq u_i$}\},
\]
and consider the edge $e_{vu}$ of $K_t$ connecting $v$ to $u$. Because
$T$ is a tree and is connected, we have that $e_{vu} \notin T$ (for
if it was an edge in $T$ then $T$ would contain a cycle). This unique
cycle of $T \cup e_{vu}$ contains $e_{vu}$, $e_{vu_i}$, and no other
edges containing $v$. Since $e_{vu}$ and $e_{vu_i}$ are the first two
edges in the cycle lexicographically, $e_{vu}$ is externally active if
and only if $u < u_i$, i.e.~$ u$ lies in the set $R_i$ from
\eqref{eq:Ri}. Hence the cardinality of $R_i$ is equal to the number
of externally active edges connecting $v$ to $T_i$, and our claim is
proved.

The full result now follows from our claim, because the sum of the $w_i$'s
obtained after deleting the vertex $v$ will now equal the number of
externally active edges that have $v$ as the lower of their two
vertices.
\end{proof}

\subsection{}\label{ss:geometric} We conclude this section by giving
more details about the geometric origin of the weight function.  We
focus on the connection to counting torus orbits, which is simpler to
explain than the connection to representations of quivers.  For more
details, we refer to \cite{supernova}.

Let $q$ be a prime power and let $\FF_{q}$ be the finite field with
$q$ elements.  Let $G (k,n)$ be the Grassmannian of $k$ planes in
$\FF_{q}^{n}$.  The maximal torus $T\simeq
(\FF_{q}^{\times})^{n-1}\subset \PGL_{n} (\FF_{q}) $ acts on $G (k,n)$
via the standard action of $\GL_{n} (\FF_{q})$ on $\FF_{q}^{n}$.  Let
us call a $T$-orbit $O \subset G (k,n)$ \emph{maximal} if its
stabilizer is trivial.

Let $O_{k,n} (q)$ be the number of maximal $T$-orbits in $G (k,n)$, as
a function of $q$.  Then we have (cf.~\cite[Theorem 3.15]{supernova})
\begin{equation}\label{eq:orbitcount}
O_{k,n} (q) = P_{p} (q) = \sum_{T\in \cT_{p}} q^{w (T)},
\end{equation}
where the sum is taken over tiered trees of type $p=(k,n-k)$.

For instance, the variety $G (1,n)$ is just the projective space
$\PP^{n-1}$, which is a toric variety.  The unique dense $T$-orbit in
$G (1,n)$ is maximal and is the only maximal orbit.  The equation
\eqref{eq:orbitcount} becomes $O_{1,n} (q) = 1$, and indeed there is
a unique maxmin tree with $n$ vertices and one maximum
(cf.~Figure \ref{fig:4verts} for $n=4$).

For a more complicated example, consider the Grassmannian $G (2,4)$.
This is not a toric variety and has a more intricate collection of
$T$-orbits.  See for instance \cite[Figure 1]{gs}, which over $\CC$
shows the images of the different types of orbits under the moment
map; the maximal orbits are those whose moment map image is
full-dimensional.

To count the orbits over $\FF_{q}$, we can picture $2$-planes in a
$4$-dimensional vector space $V$ via lines in $\PP^{3} = \PP^{3} (V)$.
We claim there are two types of maximal $T$-orbits, depicted in Figure
\ref{fig:p3generic}).  One, as seen on the left, is given by lines in
$\PP^{3}$ not passing through any of the four $T$-fixed points in
$\PP^{3}$ (the solid black dots).  Any such line $L$ determines four
points (the grey dots), which are the intersection points of $L$ with
the $T$-fixed planes.  The $T$-action preserves the cross-ratio
$\lambda$ of these points; since the points are distinct, $\lambda$ is
an element of $\FF_{q}$ different from $0$ or $1$.  Thus there are
$q-2$ orbits of this type.  The other type of orbit corresponds to
points passing through a single $T$-fixed line and not meeting a
$T$-fixed point.  Any such line lies in a unique maximal orbit
determined by the fixed line it meets.  Thus we have $6$ such orbits.
The expression \eqref{eq:orbitcount} becomes $O_{2,4} (q) = q+4$,
which agrees with Figure \ref{fig:4verts}.

We remark that although \eqref{eq:orbitcount} enumerates maximal
$T$-orbits, there is not a bijection between maxmin trees and the
orbits.  We also remark that general tiered trees can be used to count
orbits in partial flag varieties, where the partition type of the flag
corresponds to the tier type of the tiering function.  Table \ref{tab:qpolys}
gives a short list of the corresponding orbit counts.

\begin{remark}\label{rem:tiertype}
One can show, using geometric results from \cite{supernova}, that the
polynomials $P_{p} (q)$ depend only on the partition determined by the
tier type $p$, and not the order of its parts.  Thus Table
\ref{tab:qpolys} gives a complete list of the generating functions for
weighted tiered trees with up to $6$ vertices.  We do not know an
elementary proof of this fact.
\end{remark}

\begin{figure}[htb]
\begin{center}
\includegraphics[scale=0.25]{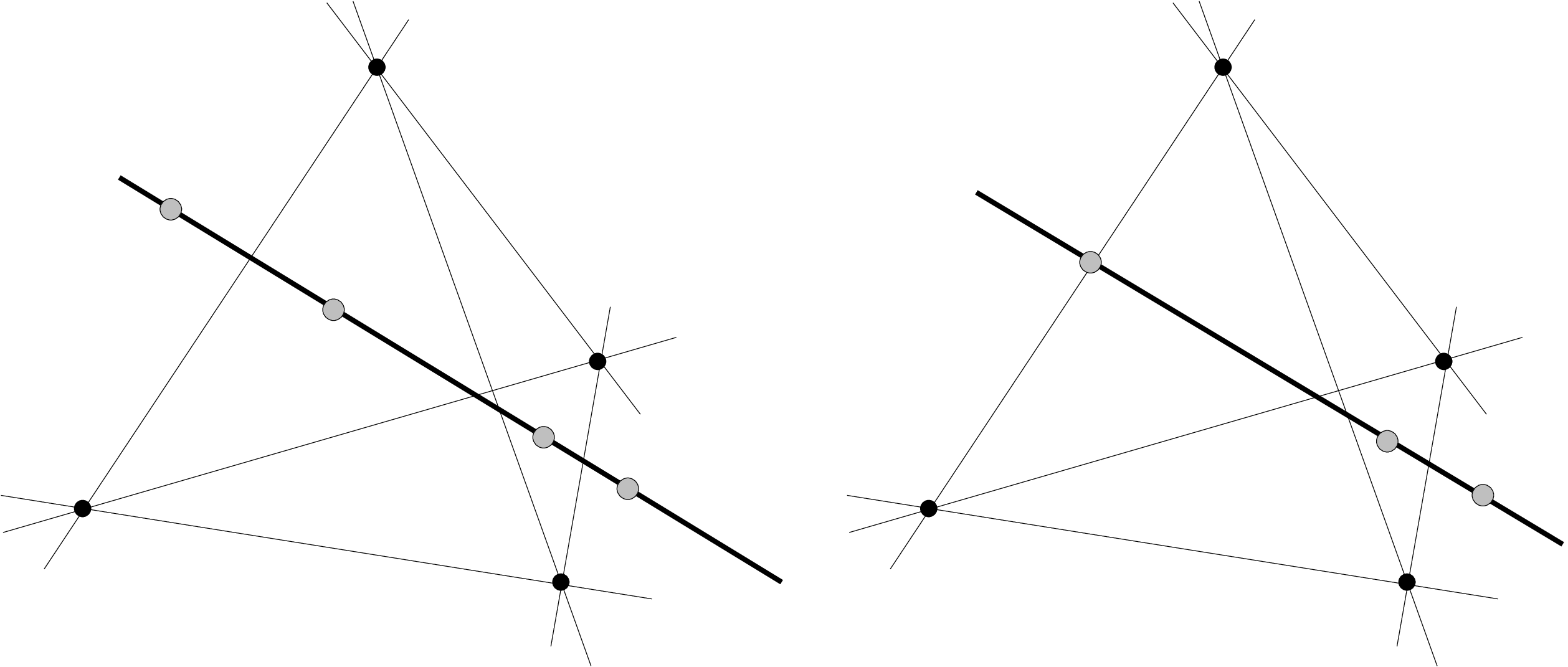}
\end{center}
\caption{The two types of maximal $T$-orbits in $G (2,4)$.  There are
$q-2$ of the left type and $6$ of the right type, giving $q+4$ altogether.\label{fig:p3generic}}
\end{figure}

\begin{table}[htb]
\begin{center}
\begin{tabular}{|l||p{4.5in}|}
\hline
Tier type $p$& $P_{p} (q) = \sum_{T\in \cT_{p}}q^{w (T)}$\\
\hline\hline
$(1, 1, 1)$&$q + 4$\\
\hline
$(2, 2)$&$q + 4$\\
$(1, 1, 2)$&$q^2 + 5q + 11$\\
$(1, 1, 1, 1)$&$q^3 + 6q^2 + 20q + 33$\\
\hline
$(2, 3)$&$q^2 + 5q + 11$\\
$(1, 1, 3)$&$q^3 + 6q^2 + 16q + 26$\\
$(1, 2, 2)$&$q^4 + 6q^3 + 22q^2 + 51q + 66$\\
$(1, 1, 1, 2)$&$q^5 + 7q^4 + 28q^3 + 78q^2 + 152q + 171$\\
$(1, 1, 1, 1, 1)$&$q^6 + 8q^5 + 35q^4 + 111q^3 + 260q^2 + 453q + 456$\\
\hline
$(2, 4)$&$q^3 + 6q^2 + 16q + 26$\\
$(3, 3)$&$q^4 + 6q^3 + 22q^2 + 51q + 66$\\
$(1, 1, 4)$&$q^4 + 7q^3 + 22q^2 + 42q + 57$\\
$(1, 2, 3)$&$q^6 + 7q^5 + 29q^4 + 85q^3 + 190q^2 + 308q + 302$\\
$(2, 2, 2)$&$q^7 + 7q^6 + 30q^5 + 97q^4 + 243q^3 + 487q^2 + 719q + 627$\\
$(1, 1, 1, 3)$&$q^7 + 8q^6 + 36q^5 + 114q^4 + 281q^3 + 549q^2 + 801q + 718$\\
$(1, 1, 2, 2)$&$q^8 + 8q^7 + 37q^6 + 127q^5 + 346q^4 + 766q^3 + 1378q^2 + 1882q + 1533$\\
$(1, 1, 1, 1, 2)$&$q^9 + 9q^8 + 45q^7 + 164q^6 + 479q^5 + 1154q^4 + 2327q^3 + 3868q^2 + \phantom{aaaaa}4957q + 3784$\\
$(1, 1, 1, 1, 1, 1)$&$q^{10} + 10q^9 + 54q^8 + 209q^7 + 649q^6 + 1681q^5 + 3691q^4 + 6921q^3 + \phantom{aaaaa}10805q^2 + 13139q + 9460$\\
\hline
\end{tabular}
\end{center}
\medskip
\caption{Tier types and the corresponding sums over weighted trees.\label{tab:qpolys}}
\end{table}

\section{Counting tiered trees}\label{s:ec}

In this section we will show how to enumerate tiered trees; our
approach closely follows that of Postnikov \cite{postnikov}. 

To this end, it will be convenient to extend the notion of tiered trees
somewhat.  Let $T$ be a tree equipped with a tiering function $t\colon
T \rightarrow \sets{m}$, where $m\geq 2$.  If $t$ is surjective, we
will say that $T$ is \emph{properly tiered}.  Otherwise, if $t$ is not
surjective, we will say that $T$ is \emph{improperly tiered} if at
least two elements of $\sets{m}$ are in the image of $t$.  (Thus we
never consider tiering functions with one element in their image.) In
particular, if $T$ is improperly tiered and has the maximum number of empty
tiers, then $T$ canonically has the structure of a maxmin tree.

Let $T_{n,m}$ be the number of properly and improperly tiered trees on
$n$ vertices with $m$ tiers.  For fixed $m$ let $T_{m}$ be the
exponential generating function of the $T_{n,m}$:

$$T_{m}(x) = \sum_{n \geq 1} T_{n+1, m} \frac{x^n}{n!}.$$ 

\begin{prop}\label{prop:fe}
The generating function $T_m(x)$ satisfies the functional equation 
\[
T_m(x) = \sum_{k=1}^{m} e^{\frac{(m-k)x}{m} \cdot (1 + T_m(x))}.
\]
\end{prop}

\begin{proof}
It is convenient to consider rooted trees instead.  Let us say that a
rooted tiered tree is \emph{$M_{i}$-rooted} if its root lies in the
$i$th tier. Then it follows from Remark \ref{rem:tiertype} that given $m$ and $n$, the total
number $M_{i,n,m}$ of $M_i$-rooted trees is
\begin{equation}\label{eq:mandt}
M_{i,n,m} = \frac{n}{m} \cdot T_{n,m}.
\end{equation}
Let us put $M_{i,1,m} =1$, and
let $M_i(t)$ be the exponential generating function for the $M_{i,n,m}$:
$$M_i(t) := \sum_{n \geq 1} M_{i,n,m} \frac{t^n}{n!}$$ 
Then using \eqref{eq:mandt}, we have
\begin{equation}\label{eq:mandt2}
M_i(t) = \frac{t}{m} ( 1 + T_{m}(t)).
\end{equation}
In particular, the generating function $M_{i} (t)$ is independent of
$i$.

Now we consider how to create a tiered tree on $n$ vertices with
tiering function taking values in $\sets{m}$.  We claim that $T_{m}
(x)$ will satisfy the relation
\begin{equation}\label{eq:fundrelation}
T_m(x) = e^{M_1(x) + M_2(x) + \dotsb + M_{m-1}(x)} + e^{M_1(x) + \dotsb +
M_{m-2}(x)} + \dotsb + e^{M_1(x)}.
\end{equation}

To see this, suppose first that $m = 2$, as in the case of maxmin
trees. Then, as shown in \cite{postnikov}, we create a maxmin tree on
$n$ vertices by taking all possible forests of rooted trees with $n-1$
vertices and with root a minimum and by attaching the $n$th vertex---a
maximum---to the root. Note that the new node must be a maximum because
it has the largest label. The total number of ways to do this is given
by $T_2(x) = e^{M_1(x)}$.

Now suppose we want to do the same for $m = 3$. We start with all
possible forests having $n-1$ vertices, but now we have a few cases to
consider. Although the $n$th node we use to connect the forests
together has the highest label, we can still form a tree in one of two
ways.  First, if we place vertex $n$ in the highest tier, it can form
a valid tiered tree by being connected to trees rooted either in the
second tier or those whose roots are in the first. The number of ways
to do this is $e^{M_1(x) + M_2(x)}$.  On the other hand, if we place
vertex $n$ in the middle tier, then it can still be used to form a
valid tiered tree, but only by connecting to trees rooted in the
lowest tier. Hence the number of ways for this to happen is given by
$e^{M_1(x)}$. Combining these cases, we find that for $m=3$ we have
$$T_3(x) = e^{M_1(x) + M_2(x)} + e^{M_1(x)}$$

If we continue this process by adding more tiers, by similar arguments
we arrive at \eqref{eq:fundrelation}.  In particular the term 
\[
e^{M_{1} (x) + \dotsb + M_{k-1} (x)} 
\]
in \eqref{eq:fundrelation} counts the contribution of those rooted
tiered trees with root $n$ in the tier $k$.

Finally, using the fact shown above that all $M_i(x)$ are equal, say
to $M_{1} (x)$, and applying \eqref{eq:mandt2} we find
\[
T_m(x) = e^{\frac{(m-1)x}{m} \cdot (1 + T_m(x))} + e^{\frac{(m-2)x}{m}
\cdot (1 + T_m(x))} +\dotsb  + e^{\frac{x}{m} \cdot (1 + T_m(x))}
\]
as desired.
\end{proof}

\begin{thm}
The total number $T_{n,m}$ of tiered trees (proper and improper) with
$n$ vertices and with $m$ tiers is given by 
\begin{equation}\label{eq:total}
T_{n,m} = \frac{1}{nm^{n-1}} \sum_{k_{i}\geq 0\text{ with }\sum k_i =
n}^{}\binom{n}{k_1, k_2,\dotsc , k_m}((m-1) \cdot k_1 + (m-2) \cdot k_2 +
\dotsb  + k_{m-1})^{n-1}.
\end{equation}
\end{thm}

We remark that substituting $m=2$ into \eqref{eq:total} yields
Postnikov's expression for the number of maxmin trees with $n$ vertices:
$$T_{n,2} = \frac{1}{n2^{n-1}} \cdot \sum_{k =
1}^{n}\binom{n}{k}k^{n-1}.$$

\begin{proof}
By Proposition \ref{prop:fe}, we have
\begin{equation}\label{eq:former}
T_{m}(x) = e^{\frac{(m-1)x}{m} \cdot (1 + T_m(x))} +
e^{\frac{(m-2)x}{m} \cdot (1 + T_m(x))} + \dotsb  + e^{\frac{x}{m} \cdot
(1 + T_m(x))}.
\end{equation}
Let $D(x) = x(1+T_{m}(x))$. Then $T_{m}(x) = {D(x)}/{x} - 1$, and
substituting in \eqref{eq:former} we find
\[
D(x)= x(e^{\frac{m-1}{m} \cdot D(x)} + e^{\frac{m-2}{m} \cdot D(x)} +\dotsb  + e^{\frac{1}{m}\cdot D(x)} + 1).
\]
We can now use Lagrange inversion to find the coefficients of
$T_{m}(x)$.  Recall that this says that a formal power series $f (x) =
a_{1}x+a_{2}x^{2} + \dotsb$, $a_{1}\not =0$ has compositional inverse
$f^{-1} (x) = b_{1}x + b_{2}x^{2} + \dotsb$ whose coefficients satisfy
\[
b_{n} = \frac{1}{n}[x^{-1}] (f (x)^{-n}),
\]
where for any formal power series $[x^{k}] (f (x))$ denotes the
coefficient of $x^{k}$ (cf.~\cite{stanley}).  Applying this to our
series, we find
\[
[x^n]D(x) = \frac{1}{nm^{n-1}}\sum \binom{n}{k_1, k_2, \dotsc , k_m}\cdot
\frac{((m-1)\cdot k_1 + (m-2)\cdot k_2 + \dotsb + k_{m-1})^{n-1}
}{(n-1)!}.
\]
But the coefficients of $D(x)$ are just those of $T_{m}(x)$ shifted,
which implies
\[
T_{n,m} = \frac{1}{nm^{n-1}}\sum \binom{n}{k_1, k_2, \dotsc , k_m}\cdot
((m-1)\cdot k_1 + (m-2)\cdot k_2 + \dotsb  + k_{m-1})^{n-1}.
\]
This completes the proof.
\end{proof}

Let $P_{n,m}$ denote the properly tiered trees on $n$ vertices with
$m$ tiers.  We can easily evaluate $P_{n,m}$ using the principle of
inclusion-exclusion:

\begin{prop}\label{prop:properly}
We have 
\begin{equation}\label{eq:prop}
P_{n,m} = \sum_{k=0}^{m-2} (-1)^{k}\binom{m}{m-k}T_{n,k}.
\end{equation}
\end{prop}

\begin{proof}
The proof is by induction on the number of tiers $m$.  We have
$P_{n,2} = T_{n,2}$ by definition.  For $m = 3$, the only possible
improperly tiered trees are those that take up two tiers, leaving the
third empty.  Thus we have $P_{n, 3} = T_{n, 3} -
\binom{3}{2}P_{n,2}$, and since $P_{n,2} = T_{n,2}$, we find $T_{n, 3}
- \binom{3}{2}T_{n,2}$.

In general, we have
\[
P_{n,m} = T_{n,m} + \sum_{k=1}^{m-2} (-1)^{k}\binom{m}{m-k}P_{n,m-k}.
\]
To pass to \eqref{eq:prop}, one substitutes the previously obtained
expressions for $P_{n,k}$, $k<m$ and uses elementary properties of
binomial coefficients.
\end{proof}
We conclude this section by giving a short table (Table
\ref{tab:tandp}) of the numbers $T_{n,m}$ and $P_{n,m}$.  We remark
that the $P_{n,m}$ can also be recovered from the polynomials in Table
\ref{tab:qpolys} by substituting $q=1$ and combining various terms.
For instance $P_{4,4} = P_{(1,1,1,1)} (1) = 60$, and $P_{4,3} =
3P_{(1,1,2)} (1) = 3\cdot 17=51$.

\begin{table}[htb]
\begin{center}
\begin{tabular}{|c|c|c|c||c|c|c|c||c|c|c|c||}
\hline
$n$&$m$&$T_{n,m}$&$P_{n,m}$&$n$&$m$&$T_{n,m}$&$P_{n,m}$&$n$&$m$&$T_{n,m}$&$P_{n,m}$\\
\hline
\hline
3&1&0&0&4&4&306&60&6&1&0&0\\
3&2&2&2&5&1&0&0&6&2&246&246\\
3&3&11&5&5&2&36&36&6&3&8868&8130\\
4&1&0&0&5&3&693&585&6&4&80496&46500\\
4&2&7&7&5&4&4304&1748&6&5&400200&83940\\
4&3&72&51&5&5&16274&1324&6&6&1414050&46620\\
\hline
\end{tabular}
\end{center}
\medskip
\caption{The numbers $T_{n,m}$ and $P_{n,m}$.\label{tab:tandp}}
\end{table}
\section{Weight zero maxmin trees and Eulerian numbers}\label{s:en}

Define a polynomial $\sT_{n}(x,q)$ by 
\[
\sT_{n} (x,q) = \sum_{T}  x^{k (T)} q^{w (T)},
\]
where $T$ ranges over all maxmin trees with $n$ vertices, and for any
tree $T$, $w (T)$ is its weight and $k (T)$ is its number of maxima.
Using Table \ref{tab:qpolys}, the first few polynomials are
\begin{align*}
\sT_{4} &= (x+x^{3}) + x^{2} (q+4),\\
\sT_{5} &= (x+x^{4}) +
(x^{2}+x^{3})(q^{2}+5q+11),\\
\sT_{6} &= (x+x^{5})+(x^{2} +x^{4})(q^3 + 6q^2 + 16q + 26) +x^{3} (q^4 + 6q^3 + 22q^2 + 51q + 66).
\end{align*}
If we set $q=0$, we find 
\begin{equation}\label{eq:eulerianpoly}
\sT_{n} (x,0) = \sum_{k=1}^{n-1} A (k-1,n-1) x^{k},
\end{equation}
where $A (k,n)$ is the Eulerian number (the number of permutations in
$S_n$ with $k$ descents).  Thus $\sT_{n} (x,0)$ is essentially the Eulerian
polynomial.

In fact, it is not difficult to prove the relationship
\eqref{eq:eulerianpoly} using generating functions.  More precisely,
the generating function of the Eulerian numbers is well known.
Postnikov \cite{postnikov} enumerated all maxmin trees on $n$
vertices, which means he computed the specialization $\sT_n (x,1)$.
One can modify his result to include the weight parameter $q$ so that
setting $q=0$ finds the generating function for the Eulerian numbers.

Our goal in this section is to give a bijective proof of
\eqref{eq:eulerianpoly}.  This has the advantage of revealing a
connection between permutations and maxmin trees that was previously
unknown.  It will also enable us later (\S\ref{s:wp}) to define a
$q$-analogue of the Eulerian numbers.

\begin{thm}\label{th:wt0}
There is a bijection between permutations in $S_{n}$ with $k$ descents
and weight $0$ maxmin trees on $(n+1)$ vertices with $(k+1)$ maxima.
\end{thm}

\begin{proof}
We define the bijection recursively.  We begin by taking the identity
permutation in $S_{n}$ to the unique maxmin tree on $(n+1)$ vertices
with $1$ maximum.  Similarly we take the longest element of $S_{n}$,
the decreasing permutation, to the unique maxmin tree on $(n+1)$
vertices with $n$ maxima.

For the rest of $S_{n}$, we regard $S_{n}$ as embedded in $S_{n+1}$ as
follows.  We take any $\pi \in S_{n}$ and represent it as an ordered
sequence of $1,\dots ,n$.  We then adjoin an $n+1$ on the right
and obtain an element of $S_{n+1}$.  In other words, we identify $S_{n}$
with the parabolic subgroup of $S_{n+1}$ generated by all simple
transpositions other than $(n,n+1)$.  We will in fact construct a
bijection between permutations in this subgroup of $S_{n+1}$ and
maxmin trees on $(n+1)$ vertices.

Thus given $\pi \in S_{n}$ represented as above as an ordered sequence
on $1,\dots ,n+1$ with $n+1$ on the right, we break $\pi$ up into the
two subsequences appearing to the left and right of $1$:  we write
\[
\pi  = \pi_{L} \cdot  1 \cdot  \pi_{R},
\]
where the operator $\cdot $ denotes concatenation.
Note that $\pi_{L}$ may be empty, but $\pi_{R}$ contains at least the
symbol $n+1$.  We further break $\pi_{L}$ into a collection of
subsequences 
\[
\pi_{L} = \pi_{1} \cdot  \pi_{2} \cdots   \pi_{l}
\]
as follows.  Let $m_{1}$ be the maximum symbol appearing in
$\pi_{L}$.  Then $\pi_{1}$ consists of the subsequence of $\pi_{L}$
from the beginning to $m_{1}$.  We then let $m_{2}$ be the maximum of
what's left after deleting $\pi_{1}$, and let $\pi_{2}$ be the
corresponding subsequence up to $m_{2}$.  We continue this process
until $\pi_{L}$ is exhausted.

The result is a collection of ordered sequences $\pi_{1}, \dots ,
\pi_{l}$ and $\pi_{R}$ with the following property: the global maximum
of each appears in the rightmost position.  By induction on the number
of vertices, we know how to build a weight $0$ maxmin tree with the
correct number of maxima and vertices for each of these subsequences
(after relabeling them to preserve their permutation patterns).  We
take the resulting forest of weight $0$ maxmin trees and join them
together into a maxmin tree for $\pi$ by joining $1$ to the
\emph{smallest maximum available} in each connected component.  The
resulting tree clearly has weight $0$.  It is easy to verify that this
construction is a bijection with the stated property.
\end{proof}
\begin{ex}\label{ex:treebuild}
Consider the element $\pi \in S_{10}$ given by the sequence
\texttt{8594673201} (we use characters \texttt{0},\dots ,\texttt{9}
for $\sets{10}$ for this example).  Note that this permutation has
five descents, which occur at \texttt{8}, \texttt{9}, \texttt{7},
\texttt{3}, and \texttt{2}.  We are thus building a maxmin tree with
six maxima.  We identify $\pi$ with the permutation in $S_{11}$
obtained by adjoining the character \texttt{A} (hexadecimal $10$) on
the right: $\pi = \mathtt{8594673201A}$.  We have
\[
\pi_{L} = \mathtt{85946732}, \quad \pi_{R} = \mathtt{1A}.
\]
The left permutation becomes broken into four pieces
\[
\pi_{L} = \mathtt{859}\cdot \mathtt{467} \cdot \mathtt{3} \cdot \mathtt{2},
\]
which means that we must connect vertex $\mathtt{0}$ to five maxmin
trees.  The resulting maxmin tree, which has six maxima, is shown in
Figure \ref{fig:maxminexample}.

\begin{figure}[htb]
\psfrag{0}{$\scriptstyle 0$}
\psfrag{1}{$\scriptstyle 1$}
\psfrag{2}{$\scriptstyle 2$}
\psfrag{3}{$\scriptstyle 3$}
\psfrag{4}{$\scriptstyle 4$}
\psfrag{5}{$\scriptstyle 5$}
\psfrag{6}{$\scriptstyle 6$}
\psfrag{7}{$\scriptstyle 7$}
\psfrag{8}{$\scriptstyle 8$}
\psfrag{9}{$\scriptstyle 9$}
\psfrag{a}{$\scriptstyle A$}
\begin{center}
\includegraphics[scale=0.5]{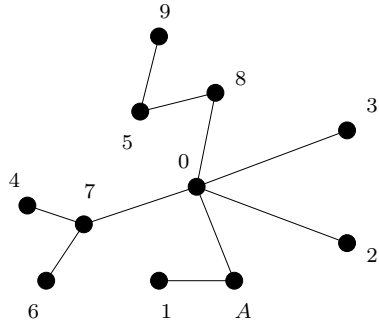}
\end{center}
\caption{The maxmin tree associated to the permutation \texttt{8594673201}.\label{fig:maxminexample}}
\end{figure}
\end{ex}

\subsection{}
We conclude this section by connecting the decomposition into
subsequences in Theorem~\ref{th:wt0} with the Stirling numbers of the
first kind.

\begin{thm}\label{stirl-second}
The number of permutations in $S_n$ that break into $k$ blocks
(subsequences) in the decomposition in Theorem~\ref{th:wt0},
disregarding the singleton block with the minimum element, equals
$n\cdot c(n-1,k)$ where $c(n-1,k)$ is the signless Stirling number of
the first kind.
\end{thm}

\begin{proof}
The signless Stirling number of the first kind, $c(n-1,k)$, counts
permutations in $S_{n-1}$ with $k$ cycles when they are written in
disjoint cycle notation.  Given such a permutation we construct $n$
different permutations in $S_n$, each of which breaks into $k$ blocks
as in the proof of Theorem~\ref{th:wt0}. Namely, we can insert $0$
after any one of the $n-1$ letters in the $k$ cycles, or else as a
cycle of its own.  Once we have done that we write the cycle
containing 0 with the 0 first and make that the last cycle in our
ordering of the cycles.  The remaining cycles we write with their
maximum letter last, and order them in decreasing order of their
maxima (followed by the cycle containing 0).  Erasing the parentheses
from the word thus constructed we obtain a permutation in $S_n$ whose
blocks are precisely the cycles we started with.  It is easy to see
that this constructs $n$ distinct permutations, and that any
permutation in $S_n$ is constructed this way.
\end{proof}

\begin{ex}
If in the permutation with cycle decomposition $(\mathtt{237})(\mathtt{418})(\mathtt{69})(\mathtt{5})$ we
insert $\mathtt{0}$ after the $\mathtt{2}$ we get the following cycles, ordered as in the
proof of Theorem~\ref{stirl-second}: $(\mathtt{69})(\mathtt{418})(\mathtt{5})(\mathtt{0372})$.  Erasing
the parentheses gives the permutation $\mathtt{6941850723}$, which decomposes
into the blocks $\mathtt{69}\cdot\mathtt{418}\cdot\mathtt{5}\cdot\mathtt{372}$.  Inserting $\mathtt{0}$ as its own
block, on the other hand, gives $(\mathtt{69})(\mathtt{418})(\mathtt{237})(\mathtt{5})(\mathtt{0})$, which yields
$\mathtt{69}\cdot\mathtt{418}\cdot\mathtt{237}\cdot\mathtt{5}$.
\end{ex}

\section{Weight zero fully tiered trees and complete nonambiguous
trees}\label{s:nat}

\subsection{} A tiered tree is called \emph{fully tiered} if all its
vertices sit on distinct levels, i.e.~it has tier type $(1,\dots ,1)$.
In this section we describe a bijection between weight $0$ fully
tiered trees and the complete nonambiguous trees defined by
Aval--Boussicault--Bouvel--Silimbani \cite{nat}.  We begin by defining
the latter objects, following \cite{nat}.

\subsection{} Let $G = \ZZ_{> 0} \times \ZZ_{> 0}$ be the grid of
nonnegative lattice points.  Given any $v\in G$, let $(X (v), Y (v))$
be its coordinates.  For later convenience, in figures we represent
points in $G$ by putting $(1,1)$ in the upper-left corner, having the
$X$ coordinate increase from left to right, and having the $Y$
coordinate increase from top to bottom.  A \emph{nonambiguous tree}
is a subset $A\subset G$ satisfying the following properties:

\begin{enumerate}
\item $(1,1)\in A$.  This is the root of the tree.
\item Let $p\in A$ be different from the root.  Then there exists a
unique point $q\in A$ with either (i) $X (q)<X (p)$ and $Y (q)=Y (p)$
or (ii) $X (q) = X (p)$ and $Y (q)<Y (p)$.  These possibilities are
exclusive: the point $q$ must satisfy one or the other.\label{it:uniqrowcol}
\item there is no empty line between two given points: if there exists
a point $p\in A$ such that $X(p) = x$ (resp. $Y (p) = y$), then for
every $x' < x$ (resp. $y' < y$) there exists $q \in A$ such that $X(q)
= x'$ (resp. $Y (q) = y'$).
\end{enumerate}

Any subset $A$ satisfying these properties determines a unique binary
tree $T (A)$ embedded in $G$ with edges running along the grid lines.
More precisely, if $p\in A$ is a point different from the root, then
the parent of $p$ is the unique point $q\in A$ preceding $p$ in the
same row or column; by condition \eqref{it:uniqrowcol} above, $q$ is
uniquely determined.  This motivates the name ``nonambiguous:'' the
tree structure on $A$ can be recovered from the vertex coordinates
even if the edges are missing.  Figure \ref{fig:4nats} shows four
different nonambiguous trees with $5$ vertices.

\begin{figure}[htb]
\begin{center}
\includegraphics[scale=0.5]{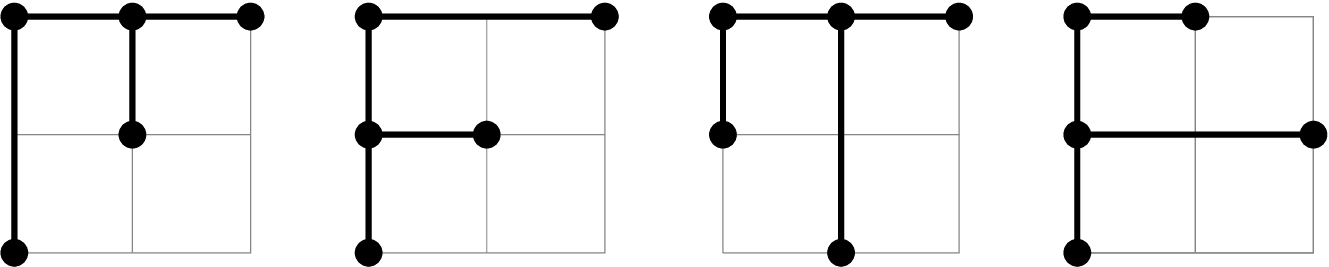}
\end{center}
\caption{Four nonambiguous trees\label{fig:4nats}}
\end{figure}

\subsection{} The trees in Figure \ref{fig:4nats} have the further
property of being \emph{complete}.  A complete nonambiguous tree is
one that is complete in the usual sense, that is, every point has $0$
or $2$ children.  Any such tree has $2k+1$ vertices, $k+1$ of which
are leaves and $k$ of which are internal.

Let $b_{k}$, $k\geq 0$ be the number of complete nonambiguous trees
with $k$ internal vertices.  The sequence $b_{k}$, which begins 
\begin{equation}\label{eq:nats}
1, 1, 4, 33, 456, 9460, 274800, 10643745, 530052880, 32995478376, \dots, 
\end{equation}
appears in OEIS \cite{oeis} as sequence
\texttt{A002190}, and is related to the logarithm of the $J$-Bessel function:
\begin{equation}\label{eq:bessel}
\sum_{k\geq 1} b_{k-1}\frac{x^{k}}{k!^{2}} = -\log J_{0} (2\sqrt{x}).
\end{equation}

According to \cite{nat}, complete nonambiguous trees provided the
first combinatorial interpretation of the sequence \texttt{A002190}, which was
originally defined through \eqref{eq:bessel}.  The following theorem
provides another combinatorial interpretation of this sequence:

\begin{thm}\label{thm:fullytiered}
There is a bijection between weight $0$ fully-tiered trees on $n$
vertices and complete nonambiguous trees with $n-1$ internal
vertices.
\end{thm}

\begin{proof}
We construct the bijection by induction.  The unique complete
nonambiguous trees with $0$ and $1$ internal vertices correspond
respectively to the unique fully tiered trees with $1$ and $2$
vertices.

A complete nonambiguous tree $T$ with $n$ internal vertices sits
inside the $(n+1)\times (n+1)$ subsquare of $\ZZ_{>0} \times
\ZZ_{>0}$.  Each row and each column contains a single leaf of
$T$. The tree $T$ determines a labeling of the $n+1$ levels of the
associated fully tiered tree $T'$ as follows: if a leaf appears at row
$i$ (numbering from the bottom of the figure to the top) and column
$j$, then at level $i$ in $T'$ we place the vertex $j$ (see Figure
\ref{fig:levels}).

Now consider erasing the first column of $T$, along with all the edges
connecting vertices in this column to nodes in greater columns.  The
result is a forest of complete nonambiguous trees, although each is
embedded with roots at different points of $\ZZ_{>0} \times \ZZ_{>0}$,
and such that their vertex labels are various subsets of $\sets{n}$.
By flattening their vertex labels in the obvious way and shifting the
roots, each of these trees corresponds to a complete nonambiguous
tree.  Thus by induction we know the fully tiered trees to which they
correspond.  This allows us to add the edges they determine to $T'$
(using the original vertex labels, not their flattened versions) to
obtain a fully tiered forest.  To complete $T'$ to a weight $0$ fully
tiered tree, we need to connect the vertex $1$ to each connected
component, and in doing so we choose the minimal vertex we can for
each component.  The resulting fully tiered tree clearly has weight
$0$ (see Figure \ref{fig:building} for a complete example).

We claim this is a bijection.  To go backwards, we reverse the
process.  We take a fully tiered tree $T'$ and delete the vertex $1$.
This gives a forest of fully tiered trees with fewer vertices.  We can
build complete nonambiguous trees from them and can embed them into
$\ZZ_{>0}\times \ZZ_{>0}$ using the disjoint sets of rows and columns
determined by their labels and by the levels on which the labels
appear.  The final step is to add the root and the necessary edges
from the first column heading right to the roots of the smaller trees.
Note that these smaller complete nonambiguous trees can be interlaced
to form a complete nonambiguous tree exactly because their labels and
levels in $T'$ form disjoint subsets.
\end{proof}

\begin{figure}[htb]
\psfrag{T}{$T$}
\psfrag{T'}{$T'$}
\psfrag{1}{$\scriptstyle 1$}
\psfrag{2}{$\scriptstyle 2$}
\psfrag{3}{$\scriptstyle 3$}
\psfrag{4}{$\scriptstyle 4$}
\psfrag{5}{$\scriptstyle 5$}
\psfrag{6}{$\scriptstyle 6$}
\psfrag{7}{$\scriptstyle 7$}
\begin{center}
\includegraphics[scale=0.3]{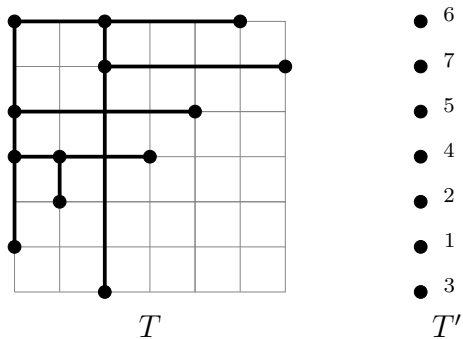}
\end{center}
\caption{Extracting levels for $T'$ from a complete nonambiguous tree
$T$.  For example, the bottommost leaf in $T$ in column $3$ puts the
label $3$ at the bottom tier of $T'$.\label{fig:levels}}
\end{figure}

\begin{figure}[htb]
\psfrag{T}{$T$}
\psfrag{T'}{$T'$}
\psfrag{1}{$\scriptstyle 1$}
\psfrag{2}{$\scriptstyle 2$}
\psfrag{3}{$\scriptstyle 3$}
\psfrag{4}{$\scriptstyle 4$}
\psfrag{5}{$\scriptstyle 5$}
\psfrag{6}{$\scriptstyle 6$}
\psfrag{7}{$\scriptstyle 7$}
\begin{center}
\includegraphics[scale=0.3]{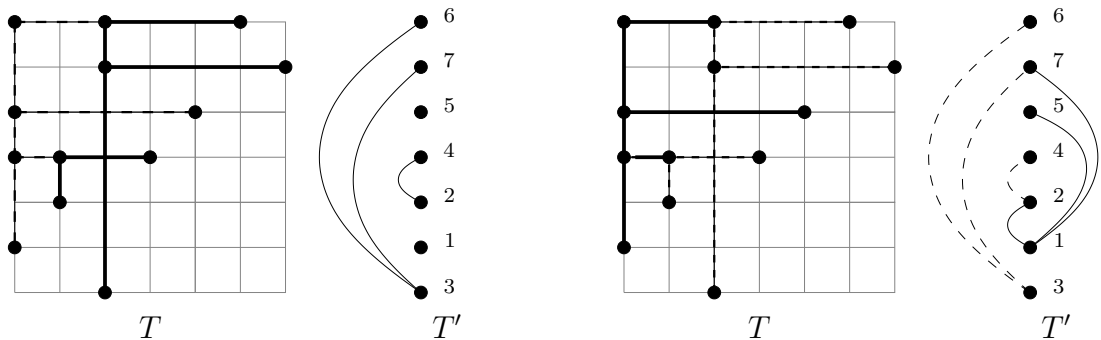}
\end{center}
\caption{Building a fully tiered tree from a complete nonambiguous
tree.  On the left, we add edges for the forest obtained by deleting
the first column.  On the right, we add edges from $1$ to the
connected components of the forest.\label{fig:building}}
\end{figure}

\section{Weights of permutations and $q$-Eulerian numbers}\label{s:wp}
Theorem \ref{th:wt0} shows that order $n$ weight $0$ maxmin trees with
$k$ maxima are in bijection with elements of $S_{n-1}$ with $k-1$
descents.  What can one say about trees of higher weight?  In fact,
the proof of Theorem \ref{th:wt0} shows that any maxmin tree of any
weight is built from an underlying permutation.  Indeed, in the
bijection between trees and permutations, one builds a weight $0$ tree
from a permutation by extracting subpermutations, building weight $0$
trees from them, and then by connecting all the smaller trees to the
minimal vertex canonically: one joins the minimal vertex to the
smaller trees by picking maxima with the smallest labels.  

Thus it is clear how to modify this construction to produce a tree of
higher weight.  At any stage, one can connect the current minimal
vertex to a larger maximum instead of choosing the smallest maximum
available.  For instance, in Example \ref{ex:treebuild}, one could
connect $\mathtt{0}$ to $\mathtt{9}$ instead of $\mathtt{8}$ to make a
tree of weight $1$.  Moreover, this is the tree of \emph{largest}
weight one can make from this permutation.  This leads to the
following definition:

\begin{defn}\label{def:permwt}
Let $\pi \in S_n$.  Then the \emph{weight} $w (\pi)$ of $\pi$ is the
largest possible weight of maxmin tree that can be constructed from
$\pi$.  More precisely, the function $w\colon S_{n}\rightarrow
\ZZ_{\geq 0}$ is defined by the following recursion:
\begin{enumerate}
\item If $\pi$ is the identity or the longest element in $S_{n}$, then we
define $w (\pi)=0$.
\item Otherwise, regard $\pi$ as an element of $S_{n+1}$ with the
symbol $n+1$ appearing on the right, as in the proof of Theorem
\ref{th:wt0}, and let $\pi_{1}$, \dots , $\pi_{l}$, $\pi_{R}$ be the
subpermutations extracted there.  Then we define
\[
w (\pi) = w (\pi_{R}) + d_{R} + \sum_{i=1}^{l} (w (\pi_{i}) + d_{i}),
\]
where $d_{i}$ (respectively, $d_{R}$) is the number of descents in
$\pi_{i}$ (respectively, $\pi_{R}$).
\end{enumerate}
\end{defn}

\begin{ex}\label{ex:permwt}
Let us consider the permutation $\pi = \mathtt{15A86290374} \in
S_{11}$.  We can complete $\pi$ to $\mathtt{15A86290374B}$ by putting
a new maximal element $\mathtt{B}$ on the right.  Extracting
subsequences around the minimal element $\mathtt{0}$, we find 
\[
\mathtt{15A}\cdot \mathtt{8629} \cdot \mathtt{0} \cdot \mathtt{374B}.
\]
Computing $w (\pi)$ proceeds as follows:
\begin{itemize}
\item $\mathtt{15A}$ flattens to $\mathtt{012}$, which is the
identity.  We have $w (\mathtt{15A}) = 0$ and $d (\mathtt{15A}) = 0$.
\item $\mathtt{8629}$ flattens to $\mathtt{2103}$, which is the
completion of the longest word.  We have $w (\mathtt{8629}) = 0$ and
$d (\mathtt{8629}) = 2$.
\item $\mathtt{374B}$ flattens to $\mathtt{0213}$.  A short
computation shows $w (\mathtt{374B}) = 1$ (compare with Figure \ref{fig:4verts}
above; the tree on the bottom right corresponds to this permutation).  We have $d (\mathtt{374B}) = 1$.
\end{itemize}
Thus $w (\mathtt{15A86290374}) = (0+0)+(0+2)+(1+1) = 4$.
\end{ex}

\subsection{}
With our definition of weight, we can define a $q$-analogue of the
Eulerian polynomials.  Put
\[
E_{n} (x,q) = \sum_{\pi \in S_{n}} x^{d (\pi)} q^{w (\pi)}.
\]
Some examples of $E_{n} (x,q)$ are given below.
\begin{align*}
E_{4} (x,q) ={}&  1+x (q^2 + 3q + 7) + x^{2} (q^2 + 4q + 6) + x^{3},\\
\begin{split}
E_{5} (x,q) ={}& 1+x (q^3 + 3q^2 + 7q + 15) + x^{2}(q^4 + 4q^3 + 11q^2 +
25q + 25) \\
&+ x^{3} ( q^3 + 5q^2 + 10q + 10) + x^{4},
\end{split}\\
\begin{split}
E_{6} (x,q) ={}&  1+
  x(q^4 + 3q^3 + 7q^2 + 15q + 31) \\
&+
  x^{2}(q^6 + 4q^5 + 11q^4 + 31q^3 + 58q^2 + 107q +
  90) \\
&+
 x^{3}(q^6 + 5q^5 + 16q^4 + 34q^3 + 76q^2 + 105q +
  65) \\
&+ 
  x^{4}(q^4 + 6q^3 + 15q^2 + 20q + 15) +
  x^{5}.
\end{split}
\end{align*}

\subsection{}
We pause to compare the polynomials $E_{n} (x,q)$ with other
$q$-Eulerian polynomials in the literature.  Let $[x]_{q} = (q^{x}-1)/
(q-1)$ and for $m\geq 1$ let 
\[
\bino{x}{m}_{q} = \frac{(q^{x}-1) (q^{x-1}-1)\dotsb (q^{x-m+1}-1)}{(q-1) (q^{2}-1)\dotsb (q^{m-1}-1)}.
\]
Carlitz \cite{carlitz} defined $q$-polynomials
$A^{*}_{n,s} (q) \in \ZZ [q]$ as follows.  First we define $A_{n,s} (q)$
by the expansion
\[
[x]_{q}^{n} = \sum_{s=1}^{n} A_{n,s} (q) \bino{x+s-1}{n}_{q},
\]
and then set $A^{*}_{n,s} (q) = q^{- (n-s) (n-s-1)/2} A_{n,s} (q)$.
The result is still an integral polynomial in $q$.  Then
Carlitz's $q$-Eulerian polynomial $E^{\text{C}}_{n}(x,q)$ is
defined by 
\[
E^{\text{C}}_{n} (x,q) = \sum_{d=0}^{n-1} A^{*}_{n,d+1} (q)x^{d}.
\]
For example,
\[
E_{3}^{\text{C}} (x,q) = x^{2}+ (2q+2)x+1,
\]
\[
E_{4}^{\text{C}} (x,q) = x^{3} + (3q^{2}+5q+3)x^{2} + (3q^{2}+5q+3)x + 1.
\]
Stanley \cite{stanley-eulerian} defined a different $q$-Eulerian polynomial as follows.  For
any permutation $\pi$, let $i (\pi)$ be the number of
\emph{inversions}:  if $\pi$ is written as an ordered list
$(a_{1},\dotsc , a_{n})$, then $i (\pi)$ counts the number of pairs
$(i,j)$ with $i<j$ and $a_{i}>a_{j}$.  Then Stanley defined 
\[
E^{\text{S}}_{n} (x,q) = \sum_{\pi \in S_{n}} x^{d (\pi)}q^{i (\pi)}. 
\]
For example,
\[
E^{\text{S}}_{3} (x,q) = x^{2}q^{3} + x (2q^{2}+2q) + 1,
\]
\[
E^{\text{S}}_{3} (x,q) = x^{3}q^{6} + x^{2} (3q^{5}+4q^{4}+3q^{3}+q^{2}) + x (q^{4}+3q^{3}+4q^{2}+3q) + 1.
\]
Finally, Shareshian--Wachs \cite{shareshian-wachs} generalized
Stanley's definition by considering four different sums of the form 
\[
E^{\text{SW}}_{n} (x,q) = \sum_{\pi \in S_{n}} x^{a (\pi)}q^{b (\pi)},
\]
where $a$ and $b$ are various statistics of permutations; one example
yields Stanley's polynomial.

\subsection{}
From these examples it is clear that our
$E_{n} (x,q)$ is quite different.  Indeed, there are many interesting
problems to investigate about the combinatorial meaning of these
polynomials.  We indicate some below, and for the present purposes
prove the following results:

\begin{thm}\label{thm:stirling}
There is a bijection between weight 0 permutations on $n$ letters with
$k-1$ descents and partitions of $\sets{n}$ into $k$ subsets.  Thus
such permutations are counted by $\stirling{n}{k}$, the Stirling
numbers of the second kind.
\end{thm}

\begin{proof}
We describe an algorithm that constructs a weight 0 permutation given
any partition of the set $\sets{n}$ (cf.~Example \ref{ex:stirling}),
and will prove that this provides our bijection.

Let $P$ be a partition of the set $\sets{n}$ into $k$ subsets.  Sort
each subset into increasing order, and denote the resulting ordered
subsets by $\pi_{1}, \pi_{2}, \dots , \pi_{k}$.  We assemble these
subpermutations $\pi_{i}$ into a permutation $\pi \in S_{n}$ as
follows.  Let $\pi_{l}$ be the sublist containing the minimum element
1, and place it in the rightmost position in $\pi$. Next, sort the
remaining $\pi_{i}$ by their respective maxima, and arrange them in
order such that their maxima decrease.  Finally, adjoin this string of
subpermutations to the left of $\pi_{l}$ to complete $\pi$. Because we
sorted each $\pi_{i}$ and ordered them with their maxima decreasing,
the final permutation $\pi $ has $k-1$ descents.

We also claim that $\pi$ has weight 0. To see this, first observe that
each $\pi_{i}$ in $\pi$ has no descents, so each one has weight
0. Moreover, each tree constructed from the $\pi_{i}$ contains only a
single maximum, so there is no choice as to how we combine them
together to construct the tree corresponding to $\pi$: the maximum in
each subtree must be connected to the global minimal element. Hence
the maxmin tree corresponding to $\pi$ has weight 0, and this is the
maximum possible weight it could have had.

Finally, notice that the correspondence described is a bijection
between partitions of $\sets{n}$ and weight 0 permutations: the steps
outlined above demonstrate that each partition determines a unique
permutation.  Conversely, given the permutation we can recover the
original partition of $\sets{n}$ by splitting the permutation at its
descent points. This completes the proof.
\end{proof}

\begin{ex}\label{ex:stirling}
Consider the partition of $\sets{10}$ into $4$ parts given by \texttt{25},
\texttt{6130}, \texttt{798}, \texttt{4}.  To build $\pi$, we sort the parts to obtain
\texttt{25}, \texttt{0136}, \texttt{789}, \texttt{4}, and then place
the \texttt{0136} block on the right of $\pi $.  The remaining blocks
are placed on the left with their maxima decreasing: \texttt{789},
\texttt{25}, \texttt{4}.  The final permutation is 
\[
\pi = \mathtt{7892540136},
\]
which has $3$ descents.  The corresponding maxmin tree is shown in
Figure~\ref{fig:stirling}.

\begin{figure}[htb]
\psfrag{0}{$\scriptstyle 0$}
\psfrag{1}{$\scriptstyle 1$}
\psfrag{2}{$\scriptstyle 2$}
\psfrag{3}{$\scriptstyle 3$}
\psfrag{4}{$\scriptstyle 4$}
\psfrag{5}{$\scriptstyle 5$}
\psfrag{6}{$\scriptstyle 6$}
\psfrag{7}{$\scriptstyle 7$}
\psfrag{8}{$\scriptstyle 8$}
\psfrag{9}{$\scriptstyle 9$}
\psfrag{a}{$\scriptstyle A$}
\begin{center}
\includegraphics[scale=0.4]{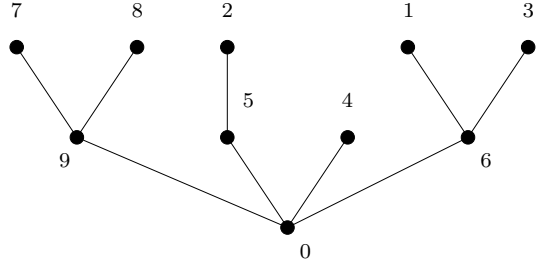}
\end{center}
\caption{The weight $0$ maxmin tree for $\pi = \mathtt{7892540136}$\label{fig:stirling}}
\end{figure}
\end{ex}

\begin{cor}\label{cor:oneandnminusone}
The coefficient of $x$ in $E_{n} (x,q)$ is $q^{n-2} +
3q^{n-3}+7q^{n-4}+ \dotsb  +(2^{n-1}-1)$, and the coefficient of
$x^{n-2}$ is $q^{n-2}+\bino{n}{1}q^{n-3} + \bino{n}{2} q^{n-4} +
\dotsb + \bino{n}{n-2}.$   
\end{cor}

\begin{proof}
We prove the first statement; the second can be proved similarly.  Let
$a (q) = a_{0}+a_{1}q + a_{2}q^{2}+ \dotsb $ be the coefficient of $x$ in $E_{n} (x,q)$, and let $b (q) =
q^{n-2} + 3q^{n-4}+7q^{n-5}+ \dotsb +(2^{n-1}-1)$.

First we observe that if $\pi\in S_{n}$ is a weight $w$ permutation
with $1$ descent, then it determines a weight $w+1$ permutation $\pi
' \in S_{n+1}$ with the same number of descents through a ``promotion''
operation: if $\pi$ is given by the ordered list $(a_{1},\dotsc
,a_{n})$, then $\pi '$ is defined by the list $ (1,a_{1}+1,\dotsc
,a_n+1)$. 

We now use this operation to determine the coefficients as follows.
By Theorem \ref{thm:stirling}, we have $a_{0} = \stirling{n}{2} =
2^{n-1}-1 = b_{0}$.  Applying promotion once, we get $a_{1} \geq
\stirling{n-1}{2} = b_{1}$.  Applying promotion a second time, we get
$a_{2} = \stirling{n-2}{2} = b_{2}$, and so on.  Hence
the difference $c (q) = a (q) - b (q)$ is a nonnegative
polynomial in $q$.  But one knows that the Euler number $A (1,n)$
equals $2^{n}-n-1$.  This implies $c (1) = 0$, which means $a (q)
= b (q)$.  Thus the coefficient of $x$ is $b (q)$ as claimed.
\end{proof}

\subsection{}
Theorem \ref{thm:stirling} addresses the constant terms of the $E_{n}
(x,q)$.  The next result shows that these polynomials are monic and
computes their degrees.

\begin{thm}\label{thm:maxwt}
The maximum weight of a permutation of length $n$ with $d$ descents
is $d(n-1-d)$.  This maximum is attained only by the permutation
$$
12\ldots(n-d-1) n(n-1)\ldots(n-d).
$$
\end{thm}

We will need the following lemma, whose simple proof we leave to the
reader:

\begin{lem}\label{lem:asc-eq-wt}
If a permutation $\pi$ ends in an ascent, and $\pi'$ is obtained by
removing the last letter from $\pi$, then $\pi'$ has the same weight
as $\pi$.
\end{lem}

\begin{proof}[Proof of Theorem \ref{thm:maxwt}]
It is clear that the permutation mentioned in the theorem has the
weight claimed.  This permutation has an empty $\pi_L$.  It is
straightforward to check that the claim is true for any $d$ and all
$n<4$.  Using that for an inductive argument, we will prove the
theorem for any $n$ and $d$.  We first show that if $\pi_L$ has two or
more components then $\pi$ has smaller weight than a permutation
obtained by merging the last two components of $\pi_L$ in a certain
way.  Then we show that a permutation with a single-component $\pi_L$
has smaller weight than the permutation obtained by merging that
$\pi_L$ and $\pi_R$ to produce a permutation of the form described in
the statement.

Suppose the last two components, $\pi_C$ and $\pi_D$, of $\pi_L$ have lengths
$(k+1)$ and $(\ell+1)$, respectively, and that their respective
numbers of descents are $a$ and $b$.  Because of how we partition, a
component cannot end in a descent (since there would have been a split
between the letters of that last descent). By Lemma~\ref{lem:asc-eq-wt},
the components have the same weight after we remove their respective
last letters.

Thus, by the inductive hypothesis, the sum of the maximum possible
weights of these two components is
\begin{equation}\label{eq2comp}
  \mathrm{wt}_2 = a(k-a) + b(\ell-b).
\end{equation}

We will compare this to the maximum weight of a permutation obtained
by using all the letters of both components to construct a single
component, which must then end in a non-descent, and have $a+b+1$
descents, since there is a descent between the original two
components.  We construct this component to have its largest letter
last (since otherwise it couldn't be a single component), preceded by
the largest $a+b+2$ letters left, in decreasing order, so as to get
$a+b+1$ descents, and with the remaining letters prepended to this in
increasing order.  An example is \texttt{12376548}.

Note that the component thus constructed will still force a break
preceding it, in the partition, since its smallest letter is clearly
smaller than the last letter of the component preceding the components
$\pi_C$ and $\pi_D$.

By Lemma~\ref{lem:asc-eq-wt}, the weight of this newly constructed
component is equal to the weight of the permutation of length
$k+\ell+1$ obtained by removing its last letter.  That weight is

\begin{equation}\label{eq1comp}
  \mathrm{wt}_1 = (a+b+1)(k+\ell-a-b-1)
\end{equation}

A straightforward computation shows that
$\mathrm{wt}_2<\mathrm{wt}_1$, since we have $a\le k-1$ and
$b\le\ell-1$ (a permutation of length $m+1$ ending in an ascent can
have at most $m-1$ descents).  This implies that if $\pi_L$ has two or
more components, these can be merged in the way described above,
preserving the number of descents in~$\pi_L$ while increasing its weight.
Repeating that process we can therefore replace~$\pi_L$ by a single
component preserving the number of descents and increasing its weight.
Thus, if $\pi$ has maximal weight, we can construct a permutation
$\pi'$ with maximal weight and $\des\pi'=\des\pi$, where $\pi'$ has a
single component preceding its minimum element.

An argument similar to the one above for the components of $\pi_L$ can now
be used to show that given $\pi'$ we can construct a permutation
$\pi''$ with the same number of descents and greater weight, whose $\pi_L$
is empty.  More precisely, if the lengths of $\pi_L$ and $\pi_R$ are $\ell$
and $k$, respectively, and their numbers of descents are $a$ and $b$,
then the maximum possible weight of $\pi_L \cdot 1 \cdot \pi_R$ is
\begin{equation}
a(\ell-2-a) + b(k-1-b),
\end{equation}
whereas the maximum possible weight of a single permutation of the
same length is (by the inductive hypothesis)
\begin{equation}
(a+b+1)(\ell+k-a-b-1).  
\end{equation}

By the inductive hypothesis the $\pi_R$ of $\pi''$ must be of the form
$$
2\ldots kn(n-1)\ldots k+1,
$$ 
and together with the prepended 1 this $\pi''$ is of the form
described in the statement of the theorem.
\end{proof}
\subsection{}
We finish by discussing some further questions about the polynomials
$E_{n} (x,q)$.

First, for fixed $d$, the coefficients of $x^{d}$ in $E_{n} (x,q)$
appear to stabilize to a fixed sequence (depending on $d$) as
$n\rightarrow \infty$.  For example, for large $n$, the coefficient of
$x^{2}$ becomes
\[
q^N + 4q^{N-1} + 11q^{N-2} + 31q^{N-3} + 65q^{N-4} + 157q^{N-5} +
298q^{N-6} + \dotsb,
\]
where $N=2 (n-3)$ is the maximal weight from Theorem \ref{thm:maxwt}.
This (conjectural) stabilization means one can define a power series
$W_{d} (t) \in \ZZ \sets{t}$ for each $d\geq 1$: if for $n$ large the
coefficient of $x^{d}$ stabilizes to
\[
q^{N} + a_{1} q^{N-1} + a_{2} q^{N-2} + \dotsb ,
\]
where now $N=d(n-1-d)$, then we put
\[
W_{d} (t) = 1 + a_{1}t +a_{2}t^{2} + \dotsb .
\]
For example,
\begin{align*}
W_{1} (t) ={}& 1 + 3t + 7t^{2}+15t^{3}+31t^{4}+63t^{5}+127t^{6}+\dotsb ,\\
W_{2} (t) ={}& 1 + 4t + 11t^{2} + 31t^{3}+65t^{4}+157t^{5}+298t^{6} +\dotsb ,\\
W_{3} (t) ={}& 1 + 5t + 16t^{2} + 41t^{3} + 112t^{4} + 244t^{5} + 542t^{6} +\dotsb, \\
W_{4} (t) ={}& 1 + 6t + 22t^{2} + 63t^{3} + 155t^{4} + 393t^{5} + 869t^{6} + \dotsb. 
\end{align*}
From this data one can make some observations.  For instance, the
coefficient of $t$ in $W_{d} (t)$ is clearly $d+2$, and the
coefficient of $t^{2}$ is clearly $(d^{2}+5d+8)/2$.  The coefficients
also enjoy many Pascal-like relations, such as $4+7=11$, $5+11=16$,
$41+22=63$.  Finally, there is a connection between the coefficients
of $W_{d}$ and the sequence \texttt{A256193} due to Alois P.~Heinz in OEIS
\cite{oeis}.  By definition, this is the triangle $T (n,k)$ of
partitions of $n$ written with two colors, where each partition
contains exactly $k$ parts of the second color.  For example (using a
prime to denote the second color), we have
\begin{align*}
T(3,0) &= 3,\text{ corresponding to } 111, 21, 3,\\
T(3,1) &= 6,\text{ corresponding to } 1'11, 11'1, 111', 2'1, 21', 3',\\
T(3,2) &= 4,\text{ corresponding to } 1'1'1, 1'11', 11'1', 2'1',\\
T(3,3) &= 1,\text{ corresponding to } 1'1'1'.
\end{align*}
A short table of the $T (n,k)$, with $k$ constant along the rows, is
given in Table~\ref{tab:tnk}.  One sees that the coefficients of
$W_{d}$ agree with the $T (n,k)$ up to a point, and then begin to
diverge.  We don't know an explanation of this connection, or why the
sequences diverge.

Finally, Eulerian polynomials have been defined for all finite Coxeter
groups (cf.~\cite{cohen}), and $q$-Eulerian polynomials have been
defined for them in some cases \cite{brenti}.  It would be interesting
to understand the proper generalization of the concepts in this paper
to finite Coxeter groups.

\begin{table}[htb]
\begin{center}
\begin{tabular}{cccccccccc}
\textbf{1}&\textbf{3}&6&12&20&35&54&86&128&$\cdots$\\
 &\textbf{1}&\textbf{4}&\textbf{11}&24&49&89&158&262&$\cdots$\\
 & &\textbf{1}&\textbf{5}&\textbf{16}&\textbf{41}&91&186&351&$\cdots$\\
 & & &\textbf{1}&\textbf{6}&\textbf{22}&\textbf{63}&\textbf{155}&342&$\cdots$
\end{tabular}
\end{center}
\medskip
\caption{The numbers $T (n,k)$\label{tab:tnk}}
\end{table}

\begin{remark}
Using similar ideas that lead to the definition of weight for
permutations, one can define a weight for complete nonambiguous trees.
To do this, one notes that each internal vertex of a complete
nonambiguous tree $T$ corresponds to an edge in the fully tiered tree
$T'$.  If $T'$ has weight $0$, then it means that during its
construction each of its edges is being connected to the minimal
vertex available.  To build a higher weight tree, we connect using
vertices with higher labels to increase the numbers $w_{i}$; on $T$,
this can be recorded by labeling its internal vertices with the
nonnegative integers $w_{i}$.  The maximal such labeling gives the
highest weight fully tiered tree $T'$ that can be built from $T$, and
thus defines the weight of $T$.  This leads to a $q$-analogue of the
numbers in the series \eqref{eq:nats}, and in fact leads to another
weight function on a subset of $S_{n}$, namely those permutations that
do not take $1$ to $n$ and $n$ to $1$ (the connection is given by the
positions of the leaves in the complete nonambiguous tree, or
equivalently the vertex labels of the tiers in the fully tiered tree).
The weight of such a permutation is then defined to be the highest
weight of any complete nonambiguous tree that can be built from it.
We have not explored these weights systematically.
\end{remark}

\bibliographystyle{amsplain_initials_eprint_doi_url}
\bibliography{tiered_trees}
\end{document}